\documentclass[10pt,reqno]{amsart}

\usepackage[utf8x]{inputenc}
\usepackage{amsmath, amsthm, amssymb,bbm,dsfont}
\usepackage{color}  
\usepackage{hyperref}

\providecommand{\abs}[1]{\left\lvert#1\right\rvert}
\providecommand{\norm}[1]{\left\lVert#1\right\rVert}

\DeclareMathOperator{\R}{\mathbb{R}}
\DeclareMathOperator{\N}{\mathbb{N}}

\newtheorem{theorem}{Theorem}[section]

\newtheorem{lemma}[theorem]{Lemma}
\newtheorem{proposition}[theorem]{Proposition}

\newtheorem{definition}[theorem]{Definition}

\theoremstyle{definition}

\newtheorem{remark}[theorem]{Remark}

\title{Non-uniqueness of Weak Solutions to the Wave Map Problem}
\author{Klaus Widmayer}

\thanks{Support of the National Science Foundation DMS-1001674 is gratefully
acknowledged.}

\address{Courant Institute of Mathematical Sciences, 251 Mercer Street, New York
10012-1185 NY, USA}
\email{klaus@cims.nyu.edu}

\begin{document}

\begin{abstract}
In this note we show that weak solutions to the wave map problem in the energy-supercritical dimension 3 are not unique. On the one hand, we find weak solutions using the penalization method introduced by Shatah \cite{MR933231} and show that they satisfy a local energy inequality. On the other hand we build on a special harmonic map to construct a weak solution to the wave map problem, which violates this energy inequality.

Finally we establish a local weak-strong uniqueness argument in the spirit of Struwe \cite{MR1692140} which we employ to show that one may even have a failure of uniqueness for a Cauchy problem with a stationary solution. We thus obtain a result analogous to the one of Coron \cite{MR1067779} for the case of the heat flow of harmonic maps.
\end{abstract}

\maketitle

\tableofcontents

\section{Introduction}
The subject under consideration in this article is the following ``wave map'' problem. For a map $\varphi:\R^{1+d}\rightarrow (M,g)$ on Minkowski space $\R^{1+d}$ to a Riemannian manifold $(M,g)$ we seek to find the critical points of the Lagrangian
\begin{equation}\label{eq:action}
 \mathcal{L}(\varphi):=\frac{1}{2}\int_{\R^{1+d}} \langle\partial^\alpha\varphi, \partial_\alpha\varphi\rangle_g \, dx dt,
\end{equation}
where we raise indices using the Minkowski metric $\eta=\textrm{diag}\{-1,1,\ldots,1\}$ and repeated indices are to be summed over. The corresponding Euler-Lagrange equations yield the Cauchy problem
\begin{equation}\label{eq:strong_wm}
 \mathcal{D}_\alpha\partial^\alpha\varphi=0,\quad
 \varphi(0)=\varphi_0,\;
 \partial_t\varphi(0)=\varphi_1,
\end{equation}
for given $(\varphi_0,\varphi_1)$ in appropriate function spaces and $\mathcal{D}_\alpha$ the induced covariant derivative on the pull-back tangent bundle $\varphi^{-1}TM$. 

For initial data $(\varphi_0,\varphi_1)\in\dot{H}^s\times\dot{H}^{s-1}$ the scaling invariant Sobolev space has differentiability exponent $s=\frac{d}{2}$.

The well-posedness theory in the subcritical and critical cases ($s\geq\frac{d}{2}$) has undergone much development, which is too extensive to be summarized exhaustively here. For the case of $M$ a sphere, a first important step beyond techniques based on Strichartz estimates was the work of Klainerman, Machedon and Selberg \cite{MR1381973,MR1452172}, where local (in time) well-posedness for regularities $s>\frac{d}{2}$ was proved using the Wave-Sobolev (or $X^{s,b}$) spaces. Subsequently Tao \cite{MR1869874} established global well-posedness for regularities $s\geq \frac{d}{2}$ in dimensions $d\geq 2$. Finally, in \cite{MR2130618} Tataru showed the global well-posedness for these regularities for a wide range of target geometries (in particular for all smooth compact manifolds) using adapted function spaces. These results rely on modern methods from harmonic analysis. In contrast, for dimensions $d\geq 4$ Shatah and Struwe \cite{MR1283026} found a way of establishing global well-posedness using gauge theory.

In the supercritical setting $s<\frac{d}{2}$, heuristically one expects ill-posedness and is thus led to the study of weak solutions in the energy class $\dot{H}^1\times L^2$. An advantage of these is their relatively easy constructability using penalization methods -- see Shatah \cite{MR933231} for the case of spheres and Freire \cite{MR1421290} for that of compact homogeneous spaces.

In a similar vein, for equivariant geometries the wave map equation reduces to an ordinary differential equation and one can explicitly construct self-similar solutions that develop singularities in finite time (see Shatah \cite{MR933231} for first such examples). It is then not difficult to show that uniqueness of solutions may fail. A more detailed characterisation using Besov spaces was given by Germain in \cite{MR2450171} and \cite{MR2494812}. In particular, these works support the intuition that stationary weak solutions (i.e.\ weak harmonic maps) should be unique amongst wave maps satisfying an energy inequality, provided they minimize the Dirichlet energy. As we shall see later (in section \ref{sec:wwm}), the present article builds on this train of thought.

Furthermore, ill-posedness has been studied more comprehensively by D'An\-co\-na and Georgiev in \cite{MR2175916}, where inter alia a non-uniqueness result with data of supercritical (but arbitrarily close to critical) regularity in dimension $d=2$ is given.

On the other hand, in dimensions $d\geq 4$ Masmoudi and Planchon \cite{MR2928107} used gauge theory to prove unconditional uniqueness\footnote{i.e.\ uniqueness without further assumptions of boundedness of higher order Lebesgue norms} of solutions in the natural class $\dot{H}^{\frac{d}{2}}\times\dot{H}^{\frac{d}{2}-1}$.

\subsection{Plan of the Article}
In the present article we study weak solutions of the Cauchy problem \eqref{eq:strong_wm} in dimension $d=3$ and $M=\mathbb{S}^2$ with the induced metric from the embedding\footnote{We adopt this external point of view for the rest of this article, in particular also for the formulation of the equation.} $\mathbb{S}^2\hookrightarrow\R^3$. This is the supercritical case for the energy norm $\dot{H}^1$: \eqref{eq:strong_wm} is subcritical in dimension $d=1$, critical in dimension $d=2$ and supercritical for $d\geq 3$.

Our main results are Theorems \ref{thm:nonuniq_wm1} and \ref{thm:nonuniq_wm2}, which assert the non-unique\-ness of weak solutions by exhibiting different weak solutions to the same Cauchy problem.

More precisely, after the necessary groundwork on energy equalities in section \ref{ssec:energy} we recall Shatah's method for constructing weak solutions by penalization (see section \ref{ssec:penal}, or \cite{MR933231}). This method makes use of energy conservation, and the solutions obtained thereby satisfy local (and global) energy inequalities. In section \ref{ssec:special_wm} we contrast this by giving the construction of a weak wave map which does not satisfy such a local energy inequality (based on an example of a harmonic map which fails to minimize the Dirichlet energy, given by H\'elein in \cite{MR1913803}). We thus establish the claimed non-uniqueness of weak solutions, Theorem \ref{thm:nonuniq_wm1}.

In section \ref{sec:weak-strong} we prove a weak-strong uniqueness result in analogy to the one of Struwe \cite{MR1692140}. This is interesting in its own right, but we employ it here to show that uniqueness of weak solutions can fail even in a scenario that allows for stationary solutions. This is our final result, Theorem \ref{thm:nonuniq_wm2}, providing an analogy to the one of Coron \cite{MR1067779} for the heat flow of harmonic maps.

It is worth noting that our strategy of proof applies in more general situations: On the one hand, whenever there is a weak harmonic map whose stress-energy tensor does not vanish we can construct a weak wave map that violates the energy inequality -- in particular, this is the case if the harmonic map is not energy minimizing. On the other hand, if we can find weak wave maps using penalization they will satisfy the energy inequality, so non-uniqueness is proved. Depending on the smoothness of the harmonic map and the geometry of the target involved, we might also have a local weak-strong uniqueness theory, in which case we can prove a result analogous to Theorem \ref{thm:nonuniq_wm2}.

\subsection{Notation}
Coordinates in Minkowski space $\R^{1+3}$ are written $(x^0,x^1,x^2,x^3)$ or $(t,x^1,x^2,x^3)$ and we will adhere to the convention of using Roman letters for indices running from $1$ to $3$ and Greek letters for those running from $0$ to $3$.

The natural setting for the energy (in-)\! equalities we work with is that of truncated (forward or backward) cones: We denote by 
\begin{equation*}
 K(\tau,p):=\{(s,x)\in\R\times\R^3:\;\abs{x-p}<\abs{s-s_0}\}
\end{equation*}
the light cone through $(\tau,p)\in\R^{1+3}$, and write $K_a^b(\tau,p):=K(\tau,p)\cap [a,b]\times\R^3$, $a<b$, for its truncation, with side $M_a^b(\tau,p)$. Its time slices are disks 
\begin{equation*}
 D(s,\tau,p):=K(\tau,p)\cap\{s\}\times\R^3
\end{equation*}
of radius $\abs{\tau-s}$, which is implicit in the notation. If we want to make the dependence on a radius $r>0$ explicit we write $D(s,\tau,p;r)$ or $D_r$.

In many cases we will drop parts of the notation which are clear from the context or not relevant for the situation at hand. An instance of this is the time coordinate of the center of a cone (denoted $\tau$ above), which we will frequently omit and thus write expressions like $K(p)$, $D(s,p;r)$, $D(s,p)$ for $(s,p)\in\R^{1+3}$. 

When referring to $K_a^b(p)$ we will often speak about a truncated cone with base $D(a,p)$, top $D(b,p)$ and side $M_a^b(p)$, or a truncated cone based at time $a$.

Integration in $\R^{1+3}$ and on hypersurfaces is performed using the standard Le\-besgue measure or the induced measure, respectively.

For a function $f:\R^{1+3}\rightarrow\R$ we shall denote by $\nabla f=(\partial_i f)_{1\leq i\leq 3}$ its gradient in the space variables, i.e.\ the vector of first order partial derivatives in the last $3$ variables of $\R^{1+3}$. In contrast, $Df=(\partial_\mu f)_{0\leq \mu\leq 3}$ shall denote the vector of all first order partial derivatives.

Unless stated otherwise, the function spaces we use are based on time slices of Minkowski space. In the case of spacetime norms we write the condition on the time coordinate first, e.g.\ $L^p(\R;X)$ denotes the space of functions $f$ on Minkowski space with $\int_{\R} \norm{f(t)}_X^p dt < \infty$, $X$ being a Banach space of functions on $\R^3$.

\vspace{10pt}
\noindent\textbf{Acknowledgement.} The author is indebted to Pierre Germain for suggesting this problem to him and very grateful for the many interesting discussions that ensued.

\section{Weak Wave Maps to a Sphere}\label{sec:wwm}
We start with a
\begin{definition}\label{def:weak_wm}
 A function $u:\R^{1+3}\rightarrow\R^3$ is a \emph{weak solution} of the wave map equation (or \emph{weak wave map}) from $\R^{1+3}$ to $\mathbb{S}^2$ provided
 \begin{enumerate}
  \item $\abs{u(t,x)}=1$ a.e.,
  \item $u:\R\rightarrow\dot{H}^1$ is weakly continuous in $t$,
  \item $\partial_t u:\R\rightarrow L^2$ is weakly continuous in $t$,
 \end{enumerate}
 and
 \begin{equation}\label{eq:weak_wm}
  \Box u + (\partial_\alpha u\cdot\partial^\alpha u) u=0
 \end{equation}
 holds in the sense of distributions.

 We say that $u$ is a \emph{local} weak solution (or \emph{local} weak wave map) if the above requirements only hold on a subset of the space-time $\R^{1+3}$.
\end{definition}

\begin{remark}\label{rem:wwm_def}
\vspace{10pt}
 \begin{enumerate}
  \item If a smooth function $u$ is a weak solution, then \eqref{eq:weak_wm} can be shown to hold in the classical sense. Hence the concept of weak solutions generalizes that of classical smooth solutions.
  \item\label{rem:lorentz_wwm} \emph{Lorentz transformations preserve weak wave maps:} Consider a linear transformation $\Lambda$ that leaves the Minkowski metric invariant (a so-called \emph{Lo\-rentz transformation}), i.e.\ $\eta(\Lambda \cdot,\Lambda \cdot)=\eta(\cdot,\cdot)$. For $u$ a (local) weak wave map one sees immediately that $u\circ\Lambda$ satisfies $1.$-$3.$ in the above definition. Furthermore, using a change of variables one checks that $u\circ\Lambda$ weakly solves equation \eqref{eq:weak_wm}.
 \end{enumerate}
\end{remark}

Associated to the wave map equation is the \emph{Cauchy problem} of finding a (weak) solution to the wave map equation that assumes given initial data:
\begin{definition}
 Given $(f,g)\in\dot{H}^1(\R^3)\times L^2(\R^3)$ with $\abs{f}=1$ and $f\cdot g=0$ we shall say that $u$ is a (local) solution to the Cauchy problem
  \begin{equation}\label{eq:CP}  
 \begin{cases}
  \begin{aligned}
   \Box u + (\partial_\alpha u\cdot\partial^\alpha u) u&=0,\\
   u(0)&=f,\\
   \partial_t u(0)&=g,
  \end{aligned}
 \end{cases}
 \end{equation}
 if $u$ is a (local) weak wave map and the equalities concerning the initial data hold in the sense of $\dot{H}^1(\R^3)$ and $L^2(\R^3)$.
\end{definition}

\subsection{The Energy}\label{ssec:energy}
An important feature of classical wave maps is the conservation of energy. In the variational framework such conservation laws arise naturally by integrating the stress-energy tensor (see section \ref{ssec:SE_tensor} below), and can be obtained in global or local versions.

\begin{definition}
 To a sufficiently regular function $u:\R^{1+3}\rightarrow\R$ we associate the \emph{energy}
 \begin{equation*}
  E(u)(t):=\frac{1}{2}\int_{\R^3}\abs{\partial_t u(t)}^2+\abs{\nabla u(t)}^2\,dx.
 \end{equation*}
 We say $u$ has \emph{finite energy} if $E(u)(t)<\infty$ for all $t$.
 
 The \emph{energy} on a disk $D(t,p)\subset \{t\}\times\R^3$ at time $t$ and centered at $p$ then is defined to be
 \begin{equation*}
  E(u;D(t,p)):=\frac{1}{2}\int_{D(t,p)}\abs{\partial_t u(t)}^2+\abs{\nabla u(t)}^2\,dx,
 \end{equation*}
 and the \emph{flux} across the sides $M_s^{t}(p)\subset [s,t]\times\R^3$ of a truncated cone (again centered at $p$) is given by
 \begin{equation*}
  Flux(u;M_s^{t}(p)):=\frac{1}{2\sqrt{2}}\int_{M_s^{t}(p)}\abs{\nabla u - \frac{x-p}{\abs{x-p}}\partial_t u}^2 \,d\sigma.
 \end{equation*}
 We note that for ease of notation we suppress the radii from the notation.
 \end{definition}

We recall that smooth wave maps satisfy the \emph{(global) energy equality}
\begin{equation}\label{eq:gl_en_smooth}
 E(u)(s)= E(u)(t)
\end{equation}
and the \emph{local energy conservation law}
\begin{equation}\label{eq:loc_en_smooth}
 E(u;D(s,p))=E(u;D(t,p))+Flux(u;M_s^{t}(p))
\end{equation}
for any $s<t$.

Nota bene: For weak solutions, in general these do \emph{not} hold. However, for the special weak wave maps constructed as below by a penalization procedure we will be able to replace them by appropriate inequalities (which in the local case will hold on almost every cone), once we fix $s$ to be the initial time -- see section \ref{ssec:penal}.

\subsubsection{The Stress-Energy Tensor}\label{ssec:SE_tensor}
Taking the variational perspective of wave maps it is natural to consider (see \cite[Chapter 2]{MR1674843}) the stress-energy tensor 
\begin{equation*}
 T_{\alpha\beta}:=\frac{1}{2}\eta_{\alpha\beta} \langle\partial^\gamma \varphi, \partial_\gamma \varphi\rangle_{\R^3} - \langle\partial_\alpha \varphi,\partial_\beta \varphi\rangle_{\R^3}
\end{equation*}
associated to a wave map $\varphi$ of locally finite energy. This tensor incorporates many symmetries of the problem and can also be used to obtain conserved quantities. In particular, it can be shown that under the assumption of sufficient regularity, extremizers of the action integral \eqref{eq:action} have a divergence-free stress-energy tensor, i.e.\ one has
\begin{equation*}
 \partial^\alpha T_{\alpha\beta}=0\;\textrm{for } 0\leq\beta\leq 3.
\end{equation*}

For later purposes we need to understand how these divergences behave under changes of coordinates in Minkowski space. We thus consider Lorentz transformations $\Lambda:\R^{1+3}\rightarrow\R^{1+3}$. A direct calculation (which we include in appendix \ref{pf:lem:SE_transf}, page \pageref{pf:lem:SE_transf}) then gives:
\begin{lemma}\label{lem:SE_transf}
 Let $T_{\alpha\beta}$ be the stress-energy tensor of a function $f:\R^4\rightarrow\R$ with finite energy. Then the stress-energy tensor $\widetilde{T}_{\alpha\beta}$ of $f\circ\Lambda$ has divergence
\begin{equation*}
 \partial^\alpha \widetilde{T}_{\alpha\beta}=\Lambda^\nu_\beta\left[\partial^{\sigma}T_{\sigma\nu} \right]\circ\Lambda.
\end{equation*}
\end{lemma}

\subsection{Construction of Weak Solutions by Penalization}\label{ssec:penal}
In \cite{MR933231}, Shatah showed how to construct finite energy solutions to the Cauchy problem \eqref{eq:CP} using a penalization method. We briefly sketch the steps involved.

Instead of trying to solve \eqref{eq:CP} directly one considers the perturbed problem
\begin{equation}\label{eq:CP_penal}
 \begin{cases}
  \begin{aligned}
  \Box u_n + n^2(\abs{u_n}^2-1)u_n&=0,\\
  u_n(0)&=f,\\
  \partial_t u_n(0)&=g,
 \end{aligned}
 \end{cases}
\end{equation}
where $\abs{f}=1$ and $f\cdot g=0$.
This problem is subcritical and we thus get global strong solutions $u_n$ for all $n\in\N$ with $u_n\in C(\R,\dot{H}^1_{loc})$, $\partial_t u_n\in C(\R,L^2_{loc})$. Moreover, we have the energy equalities
\begin{equation}\label{eq:gl_en_penal}
 \int_{\R^3}\frac{1}{2}\abs{\partial_t u_n(t)}^2+\frac{1}{2}\abs{\nabla u_n(t)}^2+n^2 F(u_n(t)) \,dx = \int_{\R^3}\frac{1}{2}\abs{\nabla f}^2 + \frac{1}{2}\abs{g}^2 \,dx,
\end{equation}
where $F:\R^{3}\rightarrow\R$, $F(x):=\frac{1}{4}(\abs{x}^2-1)^2$ is non-negative.
Since the terms on the left-hand side are bounded uniformly in $n\in\N$ and $t\in\R$ we can extract weakly star convergent subsequences to get a global weak solution (with weakly continuous first order derivatives) -- the term $n^2 F(u_n)$ enforcing that $\abs{u}=1$.

More precisely, one obtains a weak wave map $u$ satisfying the weak-star convergences $u_n \stackrel{*}{\rightharpoonup} u$ in $L^\infty_{loc}(\R,\dot{H}^1)$ and $\partial_t u_n \stackrel{*}{\rightharpoonup} \partial_t u$ in $L^\infty_{loc}(\R,L^2)$. These in turn imply the following energy \emph{in}equalities for $u$:
\begin{equation*}
 \begin{aligned}
  \int_{\R^3}\frac{1}{2}\abs{\partial_t u(t)}^2+\frac{1}{2}\abs{\nabla u(t)}^2 \,dx &\leq \liminf_{n\rightarrow\infty}\int_{\R^3}\frac{1}{2}\abs{\partial_t u_n(t)}^2+\frac{1}{2}\abs{\nabla u_n(t)}^2\\
  &\leq \int_{\R^3}\frac{1}{2}\abs{\nabla f}^2 + \frac{1}{2}\abs{g}^2 \,dx.
 \end{aligned}
\end{equation*}
Hence for any $t>0$ we have the global energy inequality
\begin{equation*}
 E(u)(0)\geq E(u)(t).
\end{equation*}

\begin{remark}\label{rem:local_penal}
 This construction may also be carried out locally: Assume that we are given initial data $(f,g)$ that satisfy $\abs{f}=1$ and $f\cdot g=0$ on a disk $D(0,p)$ at the initial time. Then we may solve \eqref{eq:CP_penal} on a (truncated) cone $C$ with base $D(0,p)$, obtaining functions $u_n$ defined on $C$. These will satisfy the energy equality
 \begin{equation}\label{eq:loc_en_penal}
   E_n(u_n;D(s,p))=E_n(u_n;D(t,p))+Flux_n(u_n;M_s^{t}(p)),
 \end{equation}
 where we have added the term $n^2 F(u_n(t))$ to both the flux and energy integrals and denoted the corresponding quantities with a subscript $n$:
 \begin{equation*}
  \begin{aligned}
   E_n(u_n;D(s,p))&=\frac{1}{2}\int_{D(s,p)}\abs{\partial_t u_n(s)}^2+\abs{\nabla u_n(s)}^2+n^2 F(u_n(s))\,dx,\\
  Flux_n(u_n;M_s^{t}(p))&=\frac{1}{2\sqrt{2}}\int_{M_s^{t}(p)}\abs{\nabla u_n - \frac{x-p}{\abs{x-p}}\partial_t u_n}^2 +n^2 F(u_n) \,d\sigma.
  \end{aligned}
 \end{equation*}

 In particular, by the positivity of the flux term we have the energy inequalities
 \begin{equation*}
  \int_{D(0,p)}\frac{1}{2}\abs{\nabla f}^2 + \frac{1}{2}\abs{g}^2 \,dx=E_n(u_n;D(0,p))\geq E_n(u_n;D(t,p))
 \end{equation*}
 for all $t$ on the truncated cone. We may hence pass to the limit as above to obtain a weak wave map $u$ on $C$.

 In particular, for data $(f,g)$ which are not globally integrable, in this way we may still construct local solutions to the Cauchy problem \eqref{eq:CP}. It is important to notice that this construction gives local solutions, but only a rough local energy inequality which disregards the flux term.
\end{remark}

Establishing a local energy inequality requires a bit more care, since the flux term itself does not behave well under the limiting process. The result is the following
\begin{lemma}[Energy Inequality for Weak Solutions Obtained by Penalization]\label{lem:en_ineq_penal}
 Let $u$ be a weak wave map obtained by penalization. Then for almost every cone based at the initial time we have
 \begin{equation}\label{eq:local_en_ineq}
  E(u;D(0,p))\geq E(u;D(t,p))+Flux(u;M_0^{t}(p)).
 \end{equation}
\end{lemma}

\begin{proof}
 By construction we have a sequence $(u_n)\subset C(\R,\dot{H}^1_{loc})$ with $(\partial_t u_n)\subset C(\R,L^2_{loc})$ whose limit is $u$. As noted above in \eqref{eq:loc_en_penal}, for these a slight modification of the global energy inequality \eqref{eq:gl_en_penal} holds:
 \begin{equation}\label{eq:loc_en_penal2}
   E_n(u_n;D(s,p))=E_n(u_n;D(t,p))+Flux_n(u_n;M_s^{t}(p)).
\end{equation}

 To prove the lemma we fix $s=0$ and want to pass to the limit as $n\rightarrow\infty$. We note that $u_n$ and $u$ have the same initial data, so that the left hand sides of \eqref{eq:local_en_ineq} and \eqref{eq:loc_en_penal2} already agree. In addition, the terms $n^2 F(u_n)$ are non-negative and uniformly bounded, so we may drop them at the cost of an inequality to obtain
 \begin{equation}\label{eq:loc_en_penal3}
   E(u;D(0,p))\geq E(u_n;D(t,p))+Flux(u_n;M_0^{t}(p)).
\end{equation}
 
 As before the energy term involving $L^2$ norms on time slices poses no problems when passing to the limit, it is the flux term we have to control. The flux term involves tangential derivatives on the sides of the cone, which a priori we cannot pass to the limit. To resolve this, the idea is to first ``mollify'' the energy inequality by integrating over cones of slightly differing sizes, thereby obtaining a flux-like quantity which behaves well with respect to the limit $n\rightarrow\infty$. Then we may pass to this limit to get back to $u$ and finally ``undo'' the mollification through a differentiation theorem.
 
 More precisely, let us fix a cone with base $D(0,p;r)$, whose radius we denote by $r>0$. Then for a small parameter $\varepsilon>0$ we consider the family of energy inequalities \eqref{eq:loc_en_penal3} for truncated cones with base $D(0,p;r+\delta)$, $\abs{\delta}<\varepsilon$ and fixed height $t$. We multiply these by $1 / \varepsilon \cdot\psi(\delta /  \varepsilon)$, where  $\psi:\R\rightarrow [0,1]$, $\int_{\R}\psi=1,$ is a smooth bump function with support in $[-1,1]$, and integrate over $\abs{\delta}<\varepsilon$. From the flux term this gives
 \begin{equation*}
  Flux^{\varepsilon}(u_n;M_0^{t}(p)):=\frac{1}{\varepsilon}\int_{\delta=-\varepsilon}^\varepsilon \psi(\delta / \varepsilon) \int_{M_0^t(p;r+\delta)} \abs{\nabla u_n - \frac{x-p}{\abs{x-p}}\partial_t u_n}^2 \,d\sigma d\delta.
 \end{equation*}
 The key observation is that this term can be viewed as a non-negative quadratic form involving integrals of first order derivatives of $u_n$ over time slices, so that by lower semicontinuity we can pass to the limit at the cost of an inequality to obtain
 \begin{equation*}
  \liminf_{n\rightarrow\infty}Flux^{\varepsilon}(u_n;M_0^{t}(p))\geq \frac{1}{\varepsilon}\int_{\delta=-\varepsilon}^\varepsilon \psi(\delta / \varepsilon) \int_{M_0^t(p;r+\delta)} \abs{\nabla u - \frac{x-p}{\abs{x-p}}\partial_t u}^2 \,d\sigma d\delta.
 \end{equation*}
 Finally, by the Lebesgue Differentiation Theorem we may let $\varepsilon\rightarrow 0$ to recover the flux term for $u$:
 \begin{equation*}
 \begin{aligned}
  \lim_{\varepsilon\rightarrow 0}\frac{1}{\varepsilon}\int_{\delta=-\varepsilon}^\varepsilon &\psi(\delta / \varepsilon) \int_{M_0^t(p;r+\delta)} \abs{\nabla u - \frac{x-p}{\abs{x-p}}\partial_t u}^2 \,d\sigma d\delta\\
  &=\int_{M_s^{t}(p)}\abs{\nabla u - \frac{x-p}{\abs{x-p}}\partial_t u}^2 \,d\sigma=2\sqrt{2}\;Flux(u;M_0^{t}(p))
 \end{aligned}
 \end{equation*}
 for almost every $p$.
 
 This same procedure leaves the energy terms unchanged, hence the lemma is proved.
\end{proof}

\subsection{A Special Weak Wave Map}\label{ssec:special_wm}
Here we exhibit a special weak wave map, which is smooth everywhere except on one ray. In addition, around this ray the energy inequality fails, which will be of key importance in the non-uniqueness proof of this article.

We build up this weak wave map by applying a Lorentz transform to a certain weak harmonic map.
\subsubsection{The Harmonic Map}
We interpret weak harmonic maps as stationary weak wave maps (according to Definition \ref{def:weak_wm}) and consider the following example (see H\'elein \cite{MR1913803}, Example 1.4.19, p.44):

Let $\mathbb{B}^3=\{x\in\R^3:\;\abs{x}\leq 1\}$ be the unit ball in $\R^3$ and $\sigma$ the stereographic projection to the equatorial plane,
\begin{equation*}
 \begin{aligned}
  \sigma:\;&\mathbb{S}^2\setminus\{(0,\;0,\;-1)\}\rightarrow\R^2,\\
         &(x^1,\;x^2,\;x^3)\mapsto\frac{(x^1,\;x^2)}{1+x^3}.
 \end{aligned}
\end{equation*}
Then for $0<\lambda < \infty$
\begin{equation*}
\begin{aligned}
 v_\lambda:\; &\mathbb{B}^3\rightarrow\mathbb{S}^2, \\
 &x\mapsto \sigma^{-1}\left(\lambda\sigma\left(\frac{x}{\abs{x}}\right)\right) 
\end{aligned}
\end{equation*}
is a (weak) harmonic map in $H^1(\mathbb{B}^3,\mathbb{S}^2)$. Moreover, its stress-energy tensor
\begin{equation*}
 S_{ij}:=-\langle\partial_i v_\lambda,\partial_j v_\lambda\rangle_{\R^3} + \frac{1}{2}\delta_{ij}\sum_{k=1}^3 \langle\partial_k v_\lambda,\partial_k v_\lambda\rangle_{\R^3}
\end{equation*}
satisfies
\begin{equation*}
 \partial_i S_{ij}=V_j \delta_{x=0}, \quad\textrm{where } V=\left(\begin{array}{c}0\\0\\s(\lambda)\end{array}\right),
\end{equation*}
$\delta_{x=0}$ denotes the Dirac Delta distribution at $x=0$ and
\begin{equation*}
 s(\lambda)=\begin{cases}0, &\lambda=1,\\ -\frac{8\pi}{(\lambda^2-1)^2}(\lambda^4-4\lambda^2 \log(\lambda)-1), &\lambda\neq 1.\end{cases}
\end{equation*}

When taking the point of view of wave maps \emph{we will identify $v_\lambda$ with the stationary mapping $\R^{1+3}\ni (t,x^1,x^2,x^3)\mapsto v_\lambda(x^1,x^2,x^3)$}. The associated stress-energy tensor reads
\begin{equation*}
 S_{\alpha\beta}:=\frac{1}{2}\eta_{\alpha\beta}\langle\partial_\gamma v_\lambda, \partial^\gamma v_\lambda\rangle_{\R^3} -\langle\partial_\alpha v_\lambda,\partial_\beta v_\lambda\rangle_{\R^3}
\end{equation*}
and agrees with $S_{ij}$ for $1\leq i,j\leq 3$, but has the additional components
\begin{equation*}
 S_{00}=-\frac{1}{2}\sum_{k=1}^3 \langle\partial_k v_\lambda,\partial_k v_\lambda\rangle_{\R^3}, \quad S_{0 i}=0.
\end{equation*}

Associated with this we have the divergence equations
\begin{equation}\label{eq:stat_divSE}
 \partial^\alpha S_{\alpha[\cdotp]}=\left(\begin{array}{c}0\\0\\0\\s(\lambda)\end{array}\right)\mathds{1}_t\times\delta_{x=0}.
\end{equation}

\subsubsection{Applying the Lorentz Transformation}
To get a non-stationary wave map from the above harmonic map we need only apply a Lorentz transform to $v_\lambda$ (or change coordinates in Minkowski space, to put it differently). Thus we fix $0<\nu<1$, set $\Theta:=\left(1-\nu^2\right)^{-1/2}$ and consider the Lorentz transform $\varLambda:\R^{1+3}\rightarrow \R^{1+3}$ represented by the matrix
\begin{equation*}
\varLambda=\left(\begin{array}{llll}
\Theta & 0 & 0 & -\nu \Theta\\
0 & 1 & 0 & 0\\
0 & 0 & 1 & 0\\
-\nu \Theta & 0 & 0 & \Theta
   \end{array}\right).
\end{equation*}

The harmonic maps $v_\lambda$ then give non-stationary weak wave maps $\varphi_\lambda:=v_\lambda\circ \varLambda$ -- see also Remark \ref{rem:wwm_def}.\ref{rem:lorentz_wwm}. Explicitly we have
\begin{equation*}
 \varphi_\lambda=v_\lambda \left(\Theta(t-\nu x^3),x^1,x^2,\Theta(x^3-\nu t)\right)=v_\lambda \left(x^1,x^2,\Theta(x^3-\nu t)\right).
\end{equation*}

These functions are weak solutions to the Cauchy problem
\begin{equation}\label{eq:CP_data1}
\begin{cases}
 \begin{aligned}
  \Box \varphi + (\partial_\alpha \varphi\cdot\partial^\alpha \varphi) \varphi&=0,\\
  \varphi(0)&=v_\lambda(x^1,x^2,\Theta x^3),\\
  \partial_t\varphi(0)&=-\Theta\nu(\partial_{x^3} v_\lambda)(x^1,x^2,\Theta x^3).
 \end{aligned}
\end{cases}
\end{equation}

The key property of these $\varphi_\lambda$ is the following violation of the energy inequality.
\begin{lemma}\label{lem:no_en_ineq}
 For $\lambda > 1$ the weak wave maps $\varphi_\lambda$ do not satisfy the local energy inequality \eqref{eq:local_en_ineq} on cones intersecting the set $\{x^3=\nu t\}$.
\end{lemma}
\begin{proof}
 Fix $\lambda>1$. We use the transformation rule for the divergence of the stress-energy tensor derived in Lemma \ref{lem:SE_transf}, applied to $\varphi_\lambda=v_\lambda\circ \varLambda$. Writing $T_{\alpha\beta}$ for the stress-energy tensor of $\varphi_\lambda$ we combine \eqref{eq:stat_divSE} with the aforementioned Lemma \ref{lem:SE_transf} in order to obtain
 \begin{equation*}
  \partial^\alpha T_{\alpha[\cdotp]}=s(\lambda)\left(\begin{array}{c}-\Theta\nu\\0\\0\\ \Theta\nu\end{array}\right)\delta_{x^1=x^2=0,\,x^3=\nu t},
 \end{equation*}
 where $\delta_{x^1=x^2=0,\,x^3=\nu t}$ is the Dirac Delta distribution along the line $x^3=\nu t$.
 In particular, the time component reads
 \begin{equation*}
  \partial^\alpha T_{\alpha 0}=-\Theta\nu\,s(\lambda)\,\delta_{x^1=x^2=0,\,x^3=\nu t}.
 \end{equation*}
 Integrating this divergence equation over a truncated cone\footnote{More precisely, one applies a limiting argument using test functions that converge to the characteristic function of the cone.} that intersects $\{x^3=\nu t\}$ on the base and top gives
 \begin{equation*}
  E(\varphi_\lambda;D(s,p))-E(\varphi_\lambda;D(t,p))-Flux(\varphi_\lambda;M_s^t(p))=-\Theta\nu\,s(\lambda)\,(t-s) > 0,
 \end{equation*}
 in violation of \eqref{eq:local_en_ineq}.
\end{proof}

\subsection{A First Non-uniqueness Result}\label{ssec:nonuniq_wm1}
We may now combine these observations to conclude that uniqueness of solutions to the Cauchy problem \eqref{eq:CP} fails.
\begin{theorem}\label{thm:nonuniq_wm1}
 There exist data\footnote{The restriction to local integrability arises since we start from a harmonic map which is initially only defined on a bounded domain.} $(f,g)\in\dot{H}^1_{loc}\times L^2_{loc}$ such that the Cauchy problem
 \begin{equation}\label{eq:CP2}  
 \begin{cases}
  \begin{aligned}
   \Box u + (\partial_\alpha u\cdot\partial^\alpha u) u&=0,\\
   u(0)&=f,\\
   \partial_tu(0)&=g,
  \end{aligned}
 \end{cases}
 \end{equation}
 has more than one local solution.
\end{theorem}

\begin{proof}
 We may choose the data $(f,g)$ of Cauchy problem \eqref{eq:CP_data1} from section \ref{ssec:special_wm} and restrict to the disk $D(0,0)$. On the one hand we have the solutions $\varphi_\lambda$ for $\lambda>1$, but on the other we can find a solution on the cone over $D(0,0)$ using the penalization method, as discussed in Remark \ref{rem:local_penal}. These differ, since there exist cones on which the former do not satisfy the local energy inequality (Lemma \ref{lem:no_en_ineq}), whereas the latter do (Lemma \ref{lem:en_ineq_penal}).
\end{proof}

\begin{remark}\label{rem:general1}
 We note that this approach builds on two key ingredients, which one may hope to have in more general settings: The existence of a weak harmonic map whose stress-energy tensor is not divergence-free, and the construction of solutions using the penalization method. In particular, the former holds for weak harmonic maps that do not minimize the Dirichlet energy (see H\'elein \cite[Chapter 1]{MR1913803}) and the latter can be carried out for compact homogeneous spaces as well (see Freire \cite{MR1421290}).
\end{remark}

\section{Local Weak-Strong Uniqueness}\label{sec:weak-strong}
In this section we turn to the question of local uniqueness of weak wave maps. As this article shows, in general this fails.

At this point, apart from their existence there is not much that can be said about weak wave maps. Their implicit construction in the previous paragraph does not give much information about their qualitative properties. In particular, regularity of the initial data will not be passed on to the solution and there may be several different wave maps with the same initial data. However, as we show in this section, if the initial data are smooth and if there exists a smooth solution that satisfies these data, then it is unique among all weak solutions \emph{which satisfy the local energy inequality \eqref{eq:local_en_ineq}}.

We draw our inspiration from the weak-strong uniqueness result \cite{MR1692140} by Struwe for the setting of global energy inequalities.

\begin{proposition}[Weak-Strong Uniqueness for Wave Maps]\label{prop:weak-strong}
 Suppose we are given a smooth, classical wave map $u$ and a weak wave map $v$, which satisfies the local energy inequality
 \begin{equation}\label{eq:loc_en_ineq2}
  E(v;D(0,p))\geq E(v;D(t,p))+Flux(v;M_0^{t}(p))
 \end{equation}
 on a cone\footnote{As we saw in Lemma \ref{lem:en_ineq_penal}, for weak solutions obtained by penalization this holds true for almost every cone.}. If $u$ and $v$ have the same initial data on its base, then they agree in the whole cone.
\end{proposition}

\begin{remark}
 For this to hold weaker regularity assumptions suffice: If $u$ and $v$ are weak solutions, the minimal extra regularity we need for $u$ is
  \begin{equation*}
   Du\in L^\infty_{loc}(\R;L_{loc}^{3+\epsilon})\cap L^1_{loc}(\R;L_{loc}^\infty\cap\dot{H}_{loc}^{\frac{3}{2}+\varepsilon})
  \end{equation*}
  for some $\varepsilon>0$ (see \cite{MR1692140}, adapted to three space dimensions). We note that this is just above the regularity given by the scaling of $Du$.
\end{remark}

\begin{proof}
 Let us fix a backward light cone in Minkowski space $\R^{1+3}$ with base $D_R$ for some $R>0$, at time $t=0$. For $0<T<R$ we consider its truncation at time $T$ and denote by $M_T$ the side of that truncated cone. As usual we omit to write the center of the base resp. top balls ($D_R$ resp. $D_{R-T}$) explicitly.

 \textbf{Notation} We simplify the notation by dropping the domains in the flux and energy expressions, i.e.\ for the energy at time $t$ we write
 \begin{equation*}
  E_t(\phi):=\frac{1}{2}\int_{D_{R-t}}\abs{D\phi(t)}^2 \,dx
 \end{equation*}
 and for the flux until time $t$
 \begin{equation*}
  Flux_t(\phi):=\frac{1}{2}\int_{M_t}Q(\phi,\phi)\,d\sigma \geq 0,
 \end{equation*}
 where the positive, quadratic \emph{flux form} $Q$ is the inner product
 \begin{equation*}
  Q(\phi,\psi):=\frac{1}{\sqrt{2}}\langle\nabla\phi - \phi_t \vec{n},\nabla\psi - \psi_t \vec{n}\rangle_{\R^3},
 \end{equation*}
 $\vec{n}$ denoting the (space) unit normal to $D_{R-t}$. With these conventions the local energy inequality reads 
 \begin{equation*}
  E_0(\phi)\geq E_T(\phi)+Flux_T(\phi).
 \end{equation*}
 We will also write the wave maps equation\footnote{This is the geometric notation with $A$ being the second fundamental form of $\mathbb{S}^2$.} as $\Box u=A(u)(Du,Du)$, where $A(\phi)=-\phi$ and $(D\phi,D\psi)=\partial_\alpha\phi\, \partial^\alpha\psi$. 

 \textbf{Idea} We will prove the proposition by considering the difference $w:=u-v$ and showing that the set $\{0\leq t\leq R: \;E_t(w)=0\}$ is non-empty, open and closed in $[0,R]$. It is non-empty since by assumption $u$ and $v$ agree at the initial time, hence $E_0(w)=0$. To show closedness let us suppose $(t_n)_{n\in\N}$ is a sequence with $E_{t_n}(w)=0$ and $t_n\rightarrow t_0$, so that $u(t_n)\rightharpoonup u(t_0)$ in $\dot{H}^1$ and $\partial_t u(t_n)\rightharpoonup \partial_t u(t_0)$ in $L^2$ (see Definition \ref{def:weak_wm}). Since these norms are convex and strongly continuous they are weakly lower semicontinuous, which gives closedness: $\norm{u_{t_0}}_{\dot{H}^1}\leq \liminf_{n\rightarrow\infty} \norm{u_{t_n}}_{\dot{H}^1}=0$ and $\norm{\partial_t u_{t_0}}_{L^2}\leq \liminf_{n\rightarrow\infty} \norm{\partial_t u_{t_n}}_{L^2}=0$.
 
 So all that remains to be shown is that the set is open.

 \textbf{Proof of Openness}
 By quadratic expansion and by the energy (in-)\! equalities for $u$ and $v$ we have
  \begin{equation*}
  \begin{aligned}
  E_T(w)&=E_T(v)-E_T(u)+\int_{D_{R-t}}Du\cdot Dw \\
	&\leq E_0(v)-Flux_T(v)-E_0(u)+Flux_T(u)+\int_{D_{R-t}}Du\cdot Dw \\
	&=Flux_T(u)-Flux_T(v)+\int_{D_{R-t}}Du\cdot Dw,
  \end{aligned}
  \end{equation*}
 where the last equality holds since $E_0(v)=E_0(u)$.

 Note that
 \begin{equation*}
  \begin{aligned}
   Flux_T(u)-Flux_T(v)&=\frac{1}{2}\int_{M_T} [Q(u,u)-Q(v,v)]\ ,d\sigma \\
 		     &=\int_{M_T} Q(u,w)\, d\sigma-\frac{1}{2}\int_{M_T} Q(w,w)\, d\sigma \\
 		     &=\int_{M_T} Q(u,w)\, d\sigma-Flux_T(w),
  \end{aligned}
 \end{equation*}
 hence
 \begin{equation*}
  E_T(w) \leq -Flux_T(w)+\int_{M_T} Q(u,w)d\sigma + \int_{D_{R-t}}Du\cdot Dw \,dx.
 \end{equation*}

 We invoke the computation of Lemma \ref{lem:comp} below to see a cancellation of the second and third terms:
 \begin{equation}\label{eq:comp_used}
  \begin{aligned}
   E_T(w)&\leq -Flux_T(w)+\int_{M_T} Q(u,w)\, d\sigma +\left(\int_0^T\int_{D_{R-t}}\Box u\cdot w_t\, dxdt \right.\\
 	&\qquad + \left.\int_0^T\int_{D_{R-t}}\Box w\cdot u_t\, dxdt -\int_{M_T}Q(u,w)\,d\sigma\right) \\
 	&=-Flux_T(w)+\int_0^T\int_{D_{R-t}}\Box u\cdot w_t\, dxdt + \int_0^T\int_{D_{R-t}}\Box w\cdot u_t\, dxdt\\
 	&\leq \int_0^T\int_{D_{R-t}}\Box u\cdot w_t\, dxdt + \int_0^T\int_{D_{R-t}}\Box w\cdot u_t\, dxdt,
  \end{aligned}
 \end{equation}
 where the last inequality holds since the flux is a positive quantity.

 Next we will substitute in this the equations and observe that $A(u)\,u_t=0$ by orthogonality, so that
 \begin{equation*}
  \Box w\cdot u_t=\left(A(u)(Du,Du)-A(v)(Dv,Dv)\right)\cdot u_t=-A(v)(Dv,Dv)\cdot u_t.
 \end{equation*}
 But now we are in the same position as Struwe in the first equation of page 1186 of \cite{MR1692140}, i.e.\ we have
 \begin{equation*}
  E_T(w)\leq \int_0^T\int_{D_{R-t}}[A(u)(Du,Du)\cdot w_t - A(v)(Dv,Dv)\cdot u_t] \,dxdt,
 \end{equation*}
  and can procede as therein in order to obtain the bound
 \begin{equation*}
  E_T(w)\leq C(T)\cdot \textrm{sup}_{0<t<T}E_t(w),
 \end{equation*}
 where $C(T)$ is continuous with $C(T)\rightarrow 0$ as $T\rightarrow 0$. This concludes the proof.
 \end{proof}

All that remains to be shown is the computation used in \eqref{eq:comp_used} above:
\begin{lemma}\label{lem:comp}
 \begin{equation*}
  \int_{D_{R-t}}Du\cdot Dw=\int_0^T\int_{D_{R-t}}\Box u\cdot w_t + \Box w\cdot u_t \, dxdt-\int_{M_T}Q(u,w)\, d\sigma.
 \end{equation*}
\end{lemma}

For the proof we refer the reader to appendix \ref{pf:lem:comp}.

\section{Non-Uniqueness of Weak Wave Maps}\label{ssec:nonuniq_wm2}

With the weak-strong existence theory in hand we can strengthen our previous non-uniqueness result, Theorem \ref{thm:nonuniq_wm1}, to allow for Cauchy data that have stationary solutions:
\begin{theorem}\label{thm:nonuniq_wm2}
  There exists\footnote{The restriction to local integrability arises since we start from a harmonic map which is initially only defined on a bounded domain.} $f\in\dot{H}^1_{loc}$ such that the Cauchy problem
 \begin{equation}\label{eq:CP3}  
 \begin{cases}
  \begin{aligned}
   \Box u + (\partial_\alpha u\cdot\partial^\alpha u) u&=0,\\
   u(0)&=f,\\
   \partial_tu(0)&=0,
  \end{aligned}
 \end{cases}
 \end{equation}
 has a stationary\footnote{The proof in fact gives a family of stationary solutions.} and a non-stationary local solution.
\end{theorem}

\begin{proof}
 The proof is a combination of the local weak-strong uniqueness theory and the non-uniqueness result in section \ref{ssec:nonuniq_wm1}.

 From the proof of Theorem \ref{thm:nonuniq_wm1} we have local (in a space-time neighborhood of the origin) solutions $u$ and $\varphi_\lambda$ for $\lambda>1$ to the Cauchy problem \eqref{eq:CP2} -- the former obtained by penalization, the latter as Lorentz transform of the harmonic map in section \ref{ssec:special_wm}. Moreover, $\varphi_\lambda$ is smooth outside of the light cone $K$ centered at the origin. By construction its initial data agree with those of $u$. Hence by Proposition \ref{prop:weak-strong} $u$ and $\varphi_\lambda$ agree on every truncated cone with base at the initial time and which does not intersect $K$. In such a way we can cover all of the exterior of $K$, so by repeating this argument we can show $u=\varphi_\lambda$ everywhere outside of $K$.

 Now we undo the Lorentz transformation $\varLambda$ from section \ref{ssec:special_wm}: This leaves $K$ invariant, so we obtain two different weak wave maps $\varphi_\lambda\circ\varLambda^{-1}=v_\lambda$ and $u\circ\varLambda^{-1}$, which agree outside of $K$. In particular they agree on the initial time $t=0$, where $\partial_t (u\circ\varLambda^{-1})=\partial_t v_\lambda(0)=0$, thus solving the Cauchy problem \eqref{eq:CP3} with $f=v_\lambda$. We recall that $v_\lambda$ is stationary, thus proving the claim.
\end{proof}

\begin{remark}\label{rem:general2}
 In a similar spirit as Remark \ref{rem:general1}, the procedure of this proof need not be confined to the special case at hand. Provided we have a weak-strong uniqueness theory and a sufficiently regular harmonic map whose stress-energy tensor has non-vanishing divergence, we may employ it in a more general setting to yield a smiliar result.
\end{remark}

\newpage
\appendix

\section{Proof of the Stress-Energy Tensor Transformation Law}
\begin{proof}\label{pf:lem:SE_transf}
 By assumption $\Lambda:\R^{1+3}\rightarrow\R^{1+3}$ is a Lorentz transform, i.e.\ we have
 \begin{equation}\label{eq:lorentz_prop}
  \forall v,w\in\R^4:\; \eta(\Lambda v,\Lambda w)=\eta(v,w)\;\Leftrightarrow\eta_{\delta\theta}\Lambda^\delta_\alpha \Lambda^\theta_\beta=\eta_{\alpha\beta} \;\Leftrightarrow\eta^{\delta\theta}\Lambda_\delta^\alpha \Lambda_\theta^\beta=\eta^{\alpha\beta},
 \end{equation}
 where upper indices stand for rows and lower indices for columns of $\Lambda$ and we use the standard convention of denoting by $\eta^{\alpha\beta}$ the components of the inverse $\eta^{-1}$ of the metric $\eta$.
 
 The chain rule then reads
 \begin{equation*}
  \partial_\gamma \left(f\circ\Lambda\right)=\Lambda^\sigma_\gamma\left(\partial_\sigma f\right)\circ\Lambda,
 \end{equation*}
 so by definition we have
 \begin{equation*}
 \begin{aligned}
  \widetilde{T}_{\alpha\beta}&=\frac{1}{2}\eta_{\alpha\beta}\partial^\gamma \left(f\circ\Lambda\right) \partial_\gamma \left(f\circ\Lambda\right) - \partial_\alpha \left(f\circ\Lambda\right)\partial_\beta \left(f\circ\Lambda\right)\\
  &=\left[\frac{1}{2}\eta_{\alpha\beta}\eta^{\gamma\delta}\Lambda^\theta_\delta \left(\partial_\theta f\right)\Lambda^\sigma_\gamma\left(\partial_\sigma f\right) - \Lambda^\delta_\alpha \left(\partial_\delta f\right) \Lambda^\theta_\beta \left(\partial_\theta f\right)\right]\circ\Lambda\\
  &=\left[\frac{1}{2}\eta_{\alpha\beta}\eta^{\theta\sigma} \left(\partial_\theta f\right)\left(\partial_\sigma f\right) - \Lambda^\delta_\alpha \left(\partial_\delta f\right) \Lambda^\theta_\beta \left(\partial_\theta f\right)\right]\circ\Lambda\\
  &=\left[\frac{1}{2}\eta_{\alpha\beta}\left(\partial^\sigma f\right)\left(\partial_\sigma f\right) - \Lambda^\delta_\alpha \left(\partial_\delta f\right) \Lambda^\theta_\beta \left(\partial_\theta f\right)\right]\circ\Lambda,\\
 \end{aligned}
 \end{equation*}
 where we used that $\eta^{\gamma\delta}\Lambda^\theta_\delta\Lambda^\sigma_\gamma=\eta^{\theta \sigma}$, since $\Lambda$ is a Lorentz transform -- see \eqref{eq:lorentz_prop}.
 We compute the divergence separately for these two terms:
 \begin{equation*}
  \begin{aligned}
   \partial^\alpha\left[\eta_{\alpha\beta}\left(\partial^\sigma f\right)\left(\partial_\sigma f\right)\circ\Lambda\right]&=\eta_{\alpha\beta} \eta^{\alpha\kappa}\partial_\kappa\left[\left(\partial^\sigma f\right)\left(\partial_\sigma f\right)\circ\Lambda\right]\\
   &=\partial_\beta\left[\left(\partial^\sigma f\right)\left(\partial_\sigma f\right)\circ\Lambda\right]\\
   &=\Lambda^\nu_\beta\partial_\nu\left[\left(\partial^\sigma f\right)\left(\partial_\sigma f\right)\right]\circ\Lambda,
  \end{aligned}
 \end{equation*}
 and
 \begin{equation*}
  \begin{aligned}
   \partial^\alpha\left[\Lambda^\delta_\alpha \left(\partial_\delta f\right) \Lambda^\theta_\beta \left(\partial_\theta f\right)\circ\Lambda\right]&=\eta^{\alpha\kappa}\Lambda^\sigma_\kappa\Lambda^\delta_\alpha\Lambda^\theta_\beta \partial_\sigma\left[\left(\partial_\delta f\right) \left(\partial_\theta f\right)\right]\circ\Lambda\\
   &=\eta^{\sigma\delta}\Lambda^\theta_\beta \partial_\sigma\left[\left(\partial_\delta f\right) \left(\partial_\theta f\right)\right]\circ\Lambda\\
   &=\Lambda^\theta_\beta\partial^\delta\left[\left(\partial_\delta f\right) \left(\partial_\theta f\right)\right]\circ\Lambda
  \end{aligned}
 \end{equation*}
 using \eqref{eq:lorentz_prop} as before. After renaming dummy indices this gives the claim,
 \begin{equation*}
  \begin{aligned}
   \partial^\alpha \widetilde{T}_{\alpha\beta}&=\Lambda^\nu_\beta\left[\frac{1}{2}\partial_\nu[\left(\partial^\sigma f\right)\left(\partial_\sigma f\right)] - \partial^\delta[\left(\partial_\delta f\right) \left(\partial_\nu f\right)]\right]\circ\Lambda\\
   &=\Lambda^\nu_\beta\left[\partial^{\sigma}T_{\sigma\nu} \right]\circ\Lambda,
  \end{aligned}
 \end{equation*}
 since
 \begin{equation*}
  \partial^{\sigma}T_{\sigma\nu}=\eta^{\sigma\kappa}\partial_\kappa T_{\sigma\nu}=\frac{1}{2}\eta^{\sigma\kappa}\eta_{\sigma\nu}\partial_\kappa\left[\left(\partial^\sigma f\right)\left(\partial_\sigma f\right)\right] - \eta^{\sigma\kappa}\partial_\kappa [\left(\partial_\sigma f\right) \left(\partial_\nu f\right)].
 \end{equation*}
\end{proof}

\section{Proof of Lemma \ref{lem:comp}}
\begin{proof}\label{pf:lem:comp}
 We recall that $Dw(0)=0$ and thus
\begin{equation*}
 \int_{D_{R-t}}Du\cdot Dw \,dx=\int_0^T\frac{d}{dt}\int_{D_{R-t}}Du\cdot Dw \,dx dt.
\end{equation*}
Now
\begin{equation}\label{eq:derivatives}
 \begin{aligned}
  \frac{d}{dt}&\int_{D_{R-t}}Du\cdot Dw \,dx=-\int_{\partial D_{R-t}}Du\cdot Dw \,dS+\int_{D_{R-t}}\frac{d}{dt}(Du\cdot Dw)\,dx \\
	&=-\int_{\partial D_{R-t}}Du\cdot Dw \,dS + \int_{D_{R-t}}u_{tt}\cdot w_t+\nabla u \cdot\nabla w_t \,dx \\
	&\hspace{3.5cm} + \int_{D_{R-t}}u_t\cdot w_{tt}+\nabla u_t\cdot \nabla w \,dx \\
	&=: I + II + III
 \end{aligned}
\end{equation}

The integrals $II$ and $III$ can be treated in an analogous fashion, so we will only explicitly deal with $III$. To this end, we approximate the sharp characteristic function of the ball $D_{R-t}$ by smooth cut-offs $\chi$ converging to it. For simplicity we will denote this limiting process\footnote{Analogous limits are used in \cite{MR1283026}, pages 306-307.} by $\sim$. Then we have
\begin{equation*}
 \begin{aligned}
  \int_{D_{R-t}}u_t\cdot w_{tt}&+\nabla u_t\cdot \nabla w dx \sim \int_{\R^3}(u_t\cdot w_{tt}+\nabla u_t\cdot \nabla w)\chi \,dx \\
	&=\int w_{tt}\cdot(u_t\chi)+\nabla w\cdot\nabla (u_t\chi)-\nabla w\cdot u_t\nabla\chi \,dx \\
	&=\int \Box w\cdot (u_t\chi) \,dx - \int \nabla w\cdot u_t \nabla\chi \,dx \\
	&\sim \int_{D_{R-t}}\Box w\cdot u_t \,dx + \int_{\partial D_{R-t}}\nabla w \cdot u_t\vec{n} \,dS,
 \end{aligned}
\end{equation*}
where we used that $\nabla\chi dx \sim -\vec{n}\, dS$.

Inserting this (and the version for term $II$) into \eqref{eq:derivatives} gives
\begin{equation*}
\begin{aligned}
 \frac{d}{dt}\int_{D_{R-t}}Du\cdot Dw dx&=\int_{D_{R-t}}\Box u\cdot w_t dx + \int_{D_{R-t}}\Box w\cdot u_t dx \\
	    &\quad-\int_{\partial D_{R-t}}\underbrace{Du\cdot Dw -\nabla u \cdot w_t\vec{n}-\nabla w\cdot u_t\vec{n}}_{\sqrt{2}Q(u,w)}\,dS\\
	    &\hspace{-1.5cm}=\int_{D_{R-t}}\Box u\cdot w_t dx + \int_{D_{R-t}}\Box w\cdot u_t dx - \sqrt{2}\int_{\partial D_{R-t}}Q(u,w)dS.
\end{aligned}
\end{equation*}

Integrating this from time $0$ to $T$ gives the claim:
\begin{equation*}
\begin{aligned}
 \int_{D_{R-t}}Du\cdot Dw&=\int_0^T\int_{D_{R-t}}\Box u\cdot w_t + \Box w\cdot u_t \, dxdt \\
      &\hspace{3cm}-\sqrt{2}\int_0^T\int_{\partial D_{R-t}}Q(u,w)dSdt \\
      &\hspace{-1.5cm}=\int_0^T\int_{D_{R-t}}\Box u\cdot w_t + \Box w\cdot u_t \, dxdt-\int_{M_T}Q(u,w)d\sigma.
\end{aligned}
\end{equation*}
\end{proof}

\bibliographystyle{plain}
\bibliography{refs.bib}

\end{document}